\documentclass[a4paper,12pt]{article}
\usepackage{amsmath,amssymb}
\usepackage{amsthm}
\usepackage[dvipdfmx]{graphicx}
\usepackage{mathrsfs}
\usepackage[active]{srcltx}
\usepackage{amscd}
\usepackage{cases}
 \newtheorem{theorem}{Theorem}[section]
 \newtheorem{proposition}[theorem]{Proposition}
 \newtheorem{fact}[theorem]{Fact}
 \newtheorem{lemma}[theorem]{Lemma}
 \newtheorem{corollary}[theorem]{Corollary}
\theoremstyle{definition}
 \newtheorem{definition}[theorem]{Definition}

\numberwithin{equation}{section}
\newcommand{\R}{\boldsymbol{R}}

\newcommand{\pmt}[1]{{\begin{pmatrix} #1  \end{pmatrix}}}

\newcommand{\spann}[1]{\left\langle{#1}\right\rangle}
\newcommand{\rank}{\operatorname{rank}}
\renewcommand{\phi}{\varphi}
\renewcommand{\hom}{\operatorname{Hom}}
\renewcommand{\Gamma}{\varGamma}

\begin{document}
\begin{center}
{\large {\bf 
A note on singular points of
bundle homomorphisms from 
a tangent distribution
into a vector bundle of the
same rank
}}\\[5mm]
Kentaro Saji and Asahi Tsuchida\\[5mm]
\today
\end{center}
\begin{abstract}
We consider bundle homomorphisms between
tangent distributions and vector bundles
of the same rank.
We 
%classify generic singularities of 
%between rank two tangent distributions and
%rank two vector bundles on three manifolds, and 
study the conditions for fundamental singularities when
the bundle homomorphism is induced from a Morin map.
When the tangent distribution is the
contact structure,
we characterize singularities of the bundle homomorphism
by using the 
Hamilton vector fields.
\end{abstract}

\section{Introduction}
In \cite{suy1,coh}, the notion of coherent tangent bundle
is introduced.
It is a bundle homomorphism between the tangent bundle
and a vector bundle with the same rank with
a kind of metric.
This is a generalization of fronts and $C^\infty$-maps 
between the same dimensional manifolds.
Singular points of bundle homomorphisms $\phi:TM\to E$
are points where $\phi(p):T_pM\to E_p$
is not a bijection.
In \cite{suy1,coh}, differential geometric invariants
of singularities of bundle homomorphisms are defined and
investigated. 
On the other hand,
in \cite{index}, topological properties of singular sets
of 
bundle homomorphisms without metric 
are studied.
See \cite{a} for another kind of application of 
coherent tangent bundle.
In this paper, we consider rank $r(<m)$
tangent distributions instead of the tangent bundles
of $m$-dimensional manifolds.
Since $r<m$, 
the singularities appearing on the
bundle homomorphisms are slightly different
from the case $\phi:TM\to E$, where $\dim M=\rank E=m$,
and the case $\phi:TM\to E$, where $\dim M=\rank E=r$
either.

Let $D_1$ be a rank $r$ tangent distribution on
an $m$-dimensional manifold $M$.
Let $N$ be an $r$ dimensional manifold, and
$f:M\to N$ a map.
Then a bundle homomorphism $\phi=df:D_1\to f^*TN$ is induced
from $f$.
Singularities of $\phi$ should be related to
$D_1$ and $f$.
In this paper, we stick to our interest into the low
dimensional case, 
we study the relationships when $f$ is a Morin map,
and $D_1$ is the foliation or the contact structure
when $m=3$, $r=2$.

\section{Bundle homomorphisms and their singular point}
\subsection{Singular points of bundle homomorphisms}
With the terminology of \cite{coh},
we give definition of 
singular points of bundle homomorphisms.
Let $M$ be an $m$-dimensional manifold, and let
$D_1$ be a rank $r$ ($r<m)$ tangent distribution of $M$
namely, a subbundle of $TM$.
Let $D_2$ be a rank $r$ vector bundle over $M$,
and let $\phi:D_1\to D_2$ be a bundle homomorphism.
If the rank of the linear map $\phi_p:(D_1)_p\to (D_2)_p$
is less than $r$, then $p\in M$ is called {\it singular point\/}
of $\phi$.
We denote by $S$ the set of singular points of $\phi$.
If the rank of $\phi_p$
is $r-1$, then $p$ is called a {\it corank one singular point}.
\begin{lemma}
If\/ $p$ is a corank one singular point of\/ $\phi$.
Then there exists a neighborhood\/ $U$ of\/ $p$ and 
a section\/ $\eta_\phi\in\Gamma(D_1)$
such that if\/ $q\in S\cap U$ then\/
$(\eta_\phi)_q$ is a generator of the kernel of\/ $\phi_q$.
\end{lemma}
\begin{proof}
By taking frames of $D_1,D_2$, we consider $\phi$ 
as a matrix $M_\phi$ near $p$.
Since $\rank M_\phi(p)=r-1$, only one eigenvalue of 
$M_\phi(p)$ is zero
and the others are not zero.
Thus the eigenvalue having minimum absolute value
among the eigenvalues of $M_\phi$ is uniquely determined,
and is a real valued $C^\infty$ function near $p$.
Hence corresponding eigenvector $\eta_\phi$ is also well-defined.
We have the desired section identifying $\eta_\phi$ as a section.
\end{proof}
We call $\eta_\phi$ the {\it null section\/} of $\phi$.
We set
$$
\lambda_\phi=\det M_\phi.
$$
We call $p\in S$ is {\it non-degenerate\/}
if $d\lambda_\phi(p)\ne0$.
The notions of the null section
and the non-degeneracy is introduced in \cite{krsuy}.
\begin{lemma}
Non-degenerate singular points are of
corank one.
\end{lemma}
\begin{proof}
Let $p$ be a non-degenerate singular point.
We assume that $\rank M_\phi(p)<r-1$.
Then any $r-1$ rows of $M_\phi(p)$ are linearly dependent.
Thus
$$
(M_\phi(p))_{u_i}(p)
=
0$$
holds for any $1\leq i\leq m$,
where $(u_1,\ldots,u_m)$ is a coordinate system near $p$,
and $(~)_{u_i}=\partial/\partial u_i$.
This is a contradiction.
\end{proof}
Since $S=\{\lambda_\phi(p)=0\}$,
$S$ is a codimension one submanifold
near a non-degenerate singular point.
With the terminology of 
\cite{suy3}, 
%\cite{cripl,suy3}, 
we give the following definition:
\begin{definition}
We call a singular point 
$p\in S$ is a {\it fold-like singular point\/} if
it is corank one, and
$\eta_\phi \lambda_\phi(p)\ne0$.
We call $p\in S$ is a {\it cusp-like singular point\/} 
if
$p$ is non-degenerate and
$\eta_\phi \lambda_\phi(p)=0$ and $\eta_\phi^2 \lambda_\phi(p)\ne0$.
We call $p\in S$ is a {\it swallowtail-like singular point\/}
if
$p$ is non-degenerate, and
$\eta_\phi \lambda_\phi(p)=\eta_\phi^2 \lambda_\phi(p)=0$
and 
$
\rank 
d(\lambda_\phi, \eta_\phi \lambda_\phi, 
\eta_\phi^2 \lambda_\phi)
=3$ at $p$.
\end{definition}
Here, $\xi f$ stands for the directional derivative 
of a function $f$ by the vector field $\xi$,
and $\xi^i f$ stands for the $i$ times directional
derivative by $\xi$.
\begin{lemma}
The definitions of 
fold-like, cusp-like and swallowtail-like singular points
do not depend on the choice of the frames
of\/ $D_1$, $D_2$,
nor on the choice of the null section.
\end{lemma}
\begin{proof}
We change the frames of $D_1$ by a matrix $C_1$, 
and change the frames of $D_2$ by a matrix $C_2$.
Then $M_\phi$ is changed to $C_2^{-1}M_\phi C_1$.
Thus the
independence of the choice of frames are clear.
We show the independence of the choice of the
null section,
and the case of fold-like singular points are
also clear, since $\eta_\phi \lambda_\phi(p)$ is a directional derivative
of $(\eta_\phi)_p$.
Furthermore, the independence of the non-degeneracy is also 
clear.
We set
$\tilde\eta=a\eta_\phi+b$,
where $a$ is a non-zero function, and
$b$ is a vector field which vanishes on $S$.
Let $p\in S$ be a non-degenerate singular point.
We assume that
$d\lambda_\phi(p)\ne0$ and
$\eta_\phi \lambda_\phi(p)=0$.
Then we have
$$
\tilde\eta^2\lambda_\phi
=
a(\eta_\phi a\,\eta_\phi \lambda_\phi+a\,\eta_\phi^2\lambda_\phi
+
\eta_\phi b\lambda_\phi)
+ba\,\eta_\phi \lambda_\phi
+
a\,b\eta_\phi \lambda_\phi+b^2\lambda_\phi.
$$
Since 
$b=0$ on $S$,
$b\lambda_\phi=0$ on $S$,
and since
$\eta_\phi \lambda_\phi(p)=0$, it holds that 
$\eta_\phi b\lambda_\phi(p)=0$.
Thus
$$
\tilde\eta^2\lambda_\phi(p)
=
a^2\,\eta_\phi^2\lambda_\phi(p),
$$
we see
$\tilde\eta^2\lambda_\phi(p)=0$ is equivalent to
$\eta_\phi^2\lambda_\phi(p)=0$.

Next we assume
$d\lambda_\phi(p)\ne0$ and
$\eta_\phi \lambda_\phi(p)=\eta_\phi^2 \lambda_\phi(p)=0$.
We chose a frame $\{e_1,\ldots,e_{r-1},\eta_\phi\}$ of
$D_1$.
Then the condition for swallowtail-like singular point
is equivalent to
$$
\eta_\phi^3\lambda_\phi(p)\ne0
\quad\text{and}\quad
\rank
\pmt{
e_1\lambda_\phi&\cdots&e_{r-1}\lambda_\phi\\
e_1\eta_\phi\lambda_\phi&\cdots&e_{r-1}\eta_\phi\lambda_\phi}
(p)=2.
$$
Then we have
\begin{align*}
\tilde\eta^3\lambda_\phi
=&
a\bigg(
\eta_\phi a
\Big(\eta_\phi a\,\eta_\phi \lambda_\phi+a\,\eta_\phi^2\lambda_\phi
+\eta_\phi b\,\lambda_\phi\Big)\\
&\hspace{10mm}
+a
\Big(
\eta_\phi^2 a\,\eta_\phi \lambda_\phi
+2\eta_\phi a\,\eta_\phi^2\lambda_\phi
+a\, \eta_\phi^3\lambda_\phi+\eta_\phi^2b\lambda_\phi\Big)\\
&\hspace{10mm}
+\eta_\phi ba\,\eta_\phi \lambda_\phi
+ba\,\eta_\phi^2\lambda_\phi
+\eta_\phi a\,b\eta_\phi \lambda_\phi
+a\, \eta_\phi b \eta_\phi \lambda_\phi
+\eta_\phi b^2 \lambda_\phi
\bigg)+b*,
\end{align*}
where $*$ stand for a function.
Since $b\eta_\phi \lambda_\phi$ and $b^2\lambda_\phi$ vanish
on $S$, and by $\eta_\phi \lambda_\phi(p)=0$,
it holds that $\eta_\phi b\eta_\phi \lambda_\phi(p)
=
\eta_\phi b^2\lambda_\phi(p)=0$.
We show $\eta_\phi^2b\lambda_\phi(p)=0$.
Since $b=0$ on $S=\{\lambda_\phi=0\}$, 
and $d\lambda_\phi(p)\ne0$,
there exists a function $k$ such that
$b\lambda_\phi=k\lambda_\phi$.
Since $\eta_\phi\lambda_\phi(p)=\eta_\phi^2\lambda_\phi(p)=0$,
we see $\eta_\phi^2b\lambda_\phi(p)=0$.
Hence
$$
\tilde\eta^3\lambda_\phi(p)
=
a^3 \eta_\phi^3\lambda_\phi.
$$
On the other hand, we have
\begin{equation}\label{eq:notdep1}
\Big(e_i\tilde\eta\lambda_\phi\Big)(p)
=
\Big(e_i(a\,\eta_\phi \lambda_\phi+b\lambda_\phi)\Big)(p)
=
\Big(
e_ia\,\eta_\phi \lambda_\phi
+a\,e_i\eta_\phi \lambda_\phi+e_ib\lambda_\phi
\Big)(p).
\end{equation}
By the above,
$b\lambda_\phi=k\lambda_\phi$ holds, and 
hence the right hand side of \eqref{eq:notdep1} is
\begin{align*}
\Big(
e_ia\,\eta_\phi \lambda_\phi+a\,e_i\eta_\phi \lambda_\phi
+e_i(k\lambda_\phi)
\Big)(p)
&=
\Big(
e_ia\,\eta_\phi \lambda_\phi
+a\,e_i\eta_\phi \lambda_\phi+e_ik\,\lambda_\phi
+k\,e_i\lambda_\phi\Big)(p)\\
&=
a(p)\,e_i\eta_\phi \lambda_\phi(p)+k(p)\,e_i\lambda_\phi(p).
\end{align*}
Thus
\begin{align*}
&\rank\pmt{
e_1\lambda_\phi&\cdots&e_{r-1}\lambda_\phi\\
e_1\tilde\eta\lambda_\phi&\cdots&e_{r-1}\tilde\eta\lambda_\phi}
(p)\\
=&
\rank\pmt{
e_1\lambda_\phi&\cdots&e_{r-1}\lambda_\phi\\
a\,e_1\eta_\phi \lambda_\phi+k\,e_1\lambda_\phi&\cdots
&a\,e_{r-1}\eta_\phi \lambda_\phi+k\,e_{r-1}\lambda_\phi}
(p)\\
=&
\rank\pmt{
e_1\lambda_\phi&\cdots&e_{r-1}\lambda_\phi\\
a\,e_1\eta_\phi \lambda_\phi&\cdots
&a\,e_{r-1}\eta_\phi \lambda_\phi}
(p)
\end{align*}
shows the assertion.
\end{proof}

\subsection{Geometric interpretation of singularities}
We give geometric interpretation
of singularities of bundle homomorphisms.
If $p\in S$ is a non-degenerate singular point,
then $S$ is a codimension one submanifold.
Thus $T_pS\subset TM$ can be defined.
Let us set $m=3$ and $r=2$.
Then we have the following.
\begin{proposition}\label{prop:foldcugeom}
If\/ $p\in S$ is a fold-like singular point,
then\/ $\eta_p\not\in T_pS$.
If\/ $p\in S$ is a cusp-like singular point,
then\/ $S_2=\{p\in S\,|\,\eta_p\in T_pS\}$
is one-dimensional submanifold of\/ $S$,
and\/ $\eta_p\not\in T_pS_2$.
\end{proposition}
\begin{proof}
Since $S=\{\lambda_\phi=0\}$,
the first assertion is obvious.
By non-degeneracy, 
$\eta_\phi\lambda_\phi(p)=0$ and
$\eta_\phi^2\lambda_\phi(p)\ne0$,
it holds that
$d(\lambda_\phi,\eta_\phi\lambda_\phi)(p)\ne0$.
Since
$S_2=\{\lambda_\phi=\eta_\phi\lambda_\phi=0\}$,
$S_2$ is a one-dimensional submanifold of $S$.
The last assertion is obvious from 
$\eta_\phi^2\lambda_\phi(p)\ne0$.
\end{proof}
If $p$ is a fold-like singular point,
and
$(D_1)_p=T_pS$, then $(\eta_\phi)_p\in T_pS$.
Thus $(D_1)_p\ne T_pS$.
Let $p$ be a cusp-like singular point.
If $e_1\lambda_\phi=e_2\lambda_\phi=0$ at $p$,
then $(D_1)_p=T_pS$.
In this case, we call $p$ {\it cusp-like singular point
of tangent type}.
If $(e_1\lambda_\phi,e_2\lambda_\phi)\ne(0,0)$ at $p$,
then $(D_1)_p$ is transversal to $T_pS$.
In this case, we call $p$ {\it cusp-like singular point
of transverse type}.
The picture of $S$ and $D_1$ can be drawn in Figure \ref{fig:1}.
\begin{figure}[!ht]
\centering
\includegraphics[width=.25\linewidth]{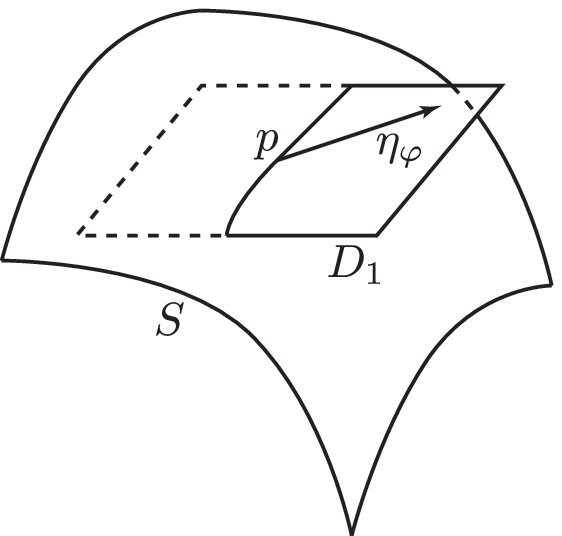}
\hspace{10mm}
\raisebox{6mm}{
\includegraphics[width=.25\linewidth]{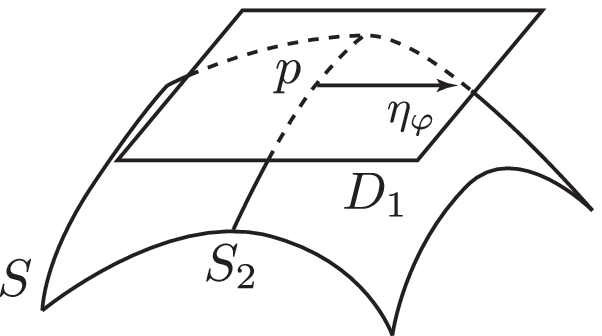}}
\hspace{10mm}
\includegraphics[width=.25\linewidth]{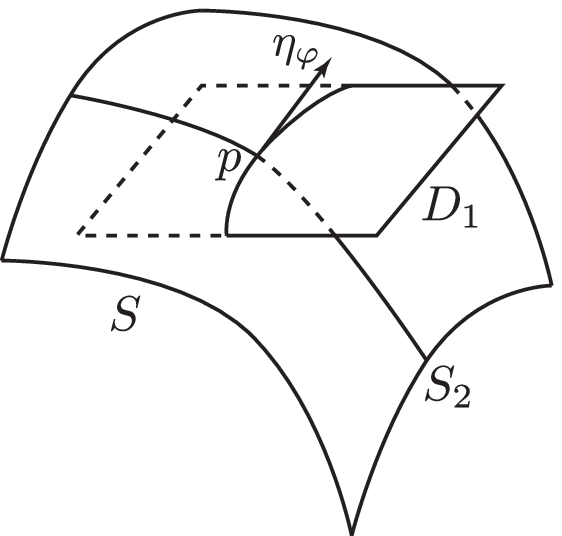}
\caption{$S$ and $D_1$ of fold-like singular point,
and cusp-like singular points
of tangent and transverse types. \label{fig:1}}
\end{figure}

If $p\in S$ is a swallowtail-like singular point,
then $S_2$ is one-dimensional submanifold of $S$.
Let $(u,v)$ be a coordinate system near $p$ of $S$.
Let $\gamma(t)=(\gamma_1(t),\gamma_2(t))$ $(\gamma(0)=p)$ 
be a parameterization
of $S_2$ with respect to $(u,v)$,
and let $\eta_{\gamma(t)}=a(t)\partial_u+b(t)\partial_v$.
Then we have the following.
\begin{proposition}\label{prop:swgeom}
Let\/ $p\in S$ is a swallowtail-like singular point.
We set 
$$
\mu(t)=\pmt{
\gamma_1(t)&a(t)\\
\gamma_2(t)&b(t)}.
$$
Under the above notation, it holds that
$$
\mu(0)=0,\quad\mu'(0)\ne0.
$$
\end{proposition}
\begin{proof}
Let $(u,v)$ be a coordinate system of $S$
satisfying $\partial_v=\eta_p$.
If we assume that $(\eta\lambda_\phi)_u=0$,
then since $S=\{\lambda_\phi=0\}$, it holds that
$(\lambda_\phi)_u=0$.
This contradicts to 
$\rank (\lambda_\phi,\eta_\phi\lambda_\phi,\eta_\phi^2\lambda_\phi)(p)=3$.
Since $(\eta\lambda_\phi)_u\ne0$,
we have a parametrization of $\gamma$ as $\gamma(t)=(\gamma_1(t),t)$.
Since $\eta_\phi^2\lambda_\phi(p)=0$, we have
$\gamma_1'(0)=0$.
On the other hand, we may take
$\eta_{\gamma(t)}=a(t)\partial_u+\partial_v$ $(a(0)=0)$.
Then $\mu(t)=\gamma_1'(t)-a(t)$.

Since $\eta_\phi^3\lambda_\phi(p)\ne0$
and $a(0)=0$, $(\lambda_\phi)_u(p)=0$,
we have
\begin{equation}\label{eq:swprf100}
3a_v(p)\lambda_{uv}(p)+\lambda_{vvv}(p)\ne0.
\end{equation}
On the other hand,
since $\eta_\phi \lambda_\phi(\gamma_1(v),v))=0$,
we have
\begin{equation}\label{eq:swprf200}
\gamma_1''(0)
=
-\dfrac{2a_v(p)\lambda_{uv}(p)+\lambda_{vvv}(p)}
{\lambda_{uv}(p)}.
\end{equation}
By \eqref{eq:swprf100} and \eqref{eq:swprf200},
we have 
$$
\mu'(0)=\gamma_1''(0)-a'(0)
=
-\dfrac{
3a_v(p)\lambda_{uv}(p)+\lambda_{vvv}(p)}
{\lambda_{uv}(p)}\ne0.
$$
\end{proof}
Like as the case of cusp-like singular point,
swallowtail-like singular point has
tangent and transverse types.
If $e_1\lambda_\phi=e_2\lambda_\phi=0$ at $p$,
then $(D_1)_p=T_pS$.
In this case, we call $p$ {\it swallowtail-like singular point
of tangent type}.
If $(e_1\lambda_\phi,e_2\lambda_\phi)\ne(0,0)$ at $p$,
then $(D_1)_p$ is transversal to $T_pS$.
In this case, we call $p$ {\it swallowtail-like singular point
of transverse type\/} (Figure \ref{fig:2}).
Ignoring arrangements of $D_1$, 
relationship of $S,S_2$ and $\eta_\phi$ is similar
to that of the 
Morin singularities of $(\R^3,0)\to(\R^3,0)$ (\cite{suy3}).
\begin{figure}[!ht]
\centering
\raisebox{6mm}{
\includegraphics[width=.3\linewidth]{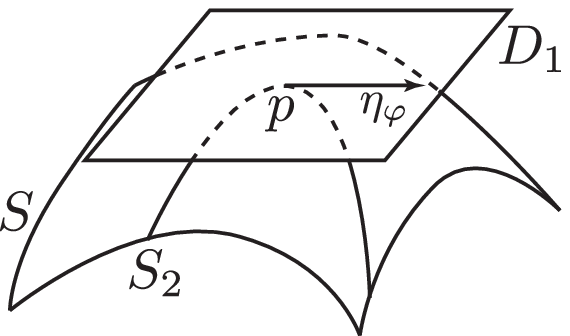}}
\hspace{15mm}
\includegraphics[width=.3\linewidth]{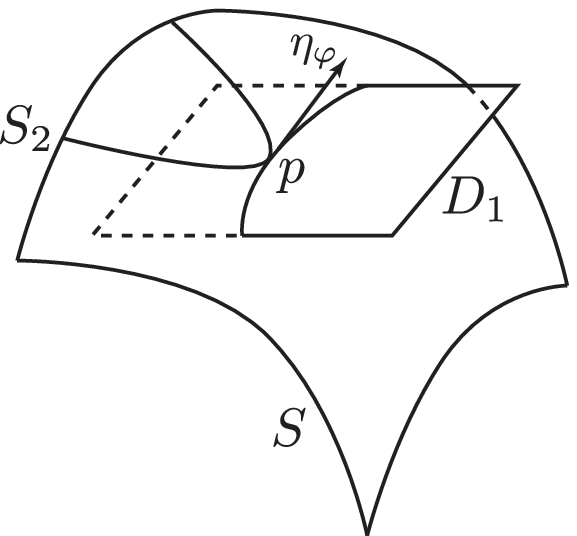}
\caption{$S$ and $D_1$ of 
swallowtail-like singular points of 
tangent and transverse types.}
\label{fig:2}
\end{figure}

\section{Generic singularities}\label{sec:generic}
We show if $m=3$ and $r=2$, then
the generic singularities of $\phi$ is
fold-like, cusp-like and swallowtail-like singular points.
The bundle homomorphism $\phi$
can be regarded as a section of the homomorphism bundle
$\hom(D_1,D_2)$. We set $E=\hom(D_1,D_2)$.
Since the set of sections $\Gamma(E)$ is a subset
of $C^\infty(M,E)$, we derive
the Whitney $C^\infty$ topology
to $\Gamma(E)$.
\begin{proposition}\label{prop:dense}
Under the above settings,
the set
$$
\{\phi\in \Gamma(E)\,|\,
\text{any\/ }p\in S
\text{ is 
fold-like, cusp-like or swallowtail-like\/}\}
$$
is dense.
\end{proposition}
For the proof of Proposition \ref{prop:dense},
we need jet transversality theorem for vector bundle
sections.
Let $J^k(\Gamma(E))$ be the subbundle
of $J^k(M,E)$ consisting of all $k$-jets
of sections.
Let $j^k:\Gamma(E)\to C^\infty(M,J^k(\Gamma(E)))$
be the jet-extension.
\begin{proposition}\label{prop:trans}
Let\/ $M$ be a manifold and let\/ $K$
be a submanifold of\/ $J^k(\Gamma(E))$. 
Then the set 
$$
\{f\in \Gamma(E) \,|\,
j^k f \text{ is transverse to\/ }K\}
$$ is residual in\/ $\Gamma(E)$,
and open dense if\/ $K$ is closed.
\end{proposition}
This is shown \cite[Theorem 2.6]{trans},
for sections of the tangent bundle.
However the proof uses the local triviality
of the tangent bundle,
so the same proof works for the case
interchanging the tangent bundle
to a general vector bundle $E$.
\begin{proof}[Proof of Proposition\/ {\rm \ref{prop:dense}}]
We set
$$
Z=\{j^3\phi(p)\in J^3(M,E)\,|\,
\phi(p)=O\}.
$$
Then $Z$ is independent of the choice of frames, and
a closed submanifold 
of codimension $4$, and $J^3(M,E)\setminus Z$
is an open submanifold.
Next, we set 
\begin{align*}
D=&\{j^3\phi(p)\in J^3(M,E)\,|\,
\det\phi(p)=0,\ 
d\det\phi(p)=(0,0,0)\}.
\end{align*}
Then $D$ is independent of the choice of frames, and
are closed
submanifolds of $J^3(M,E)\setminus Z$
of codimension $4$.
Next we consider
\begin{align*}
W_1=&\Big\{j^3\phi(p)\in J^3(M,E)\,\Big|\,
\det\phi(p)=0,\ 
\eta\det\phi(p)=0,\\
&\hspace{38mm}
\eta\eta\det\phi(p)=0,\ 
\eta\eta\eta\det\phi(p)=0\Big\},\\
W_2=&\Big\{j^3\phi(p)\in J^3(M,E)\,\Big|\,
\det\phi(p)=0,\ 
\eta\det\phi(p)=0,\\
&\hspace{38mm}
\eta\eta\det\phi(p)=0,\ 
\rank d(\det\phi,\eta\det\phi)(p)=1\Big\}.
\end{align*}
Then $W_1,W_2$ are independent of the choice of frames, and
if they are closed
submanifolds of $J^3(M,E)\setminus (Z\cup D)$
of codimension $4$.
By Proposition \ref{prop:trans},
$$
{\cal O}=\{\phi\in \Gamma(E)\,|\,
j^3\phi\text{ is transverse to }Z,\ D,\ W_1\text{ and } W_2\}.
$$
is a residual subset of $\Gamma(E)$.
So is dense.
On the other hand,
since $\dim M=3$,
$j^3\phi$ is transverse to $Z$, $D$, $W_1$ and $W_2$ is equivalent to
$j^3\phi(M)\cap (Z\cup D\cup W_1\cup W_2)=\emptyset$.
Thus, for any $\phi\in{\cal O}$ has 
only 
fold-like, cusp-like and swallowtail-like singular points as
singular points.
Thus the proof is reduced to showing the following lemma.
\end{proof}
\begin{lemma}\label{prop:submfd}
The sets\/ $W_1,W_2$ are closed
submanifolds of\/ $J^3(M,E)\setminus (Z\cup D)$
of codimension\/ $4$.
\end{lemma}
\begin{proof}
Let $p\in M$ and take a coordinate neighborhood $U$
near $p$. 
It is enough to show that $W_1,W_2$ are closed submanifolds
in $J^3(U,E|_U)\setminus (Z\cup D)$.
Since $W_1,W_2$ are independent of the choice of
coordinate system,
we choose a coordinate system $(u,v,w)$ on $U$
satisfying $\partial_w=\eta$.
Let 
$$
j^3\phi(p)
=
\pmt{j^3a(p)&j^3b(p)\\
     j^3c(p)&j^3d(p)},
$$
where $a,b,c,d$ are functions.
Then in $J^3(U,E|_U)\setminus (Z\cup D)$,
\begin{align*}
&W_1=\Biggl\{\left.\pmt{j^3a(p)&j^3b(p)\\
     j^3c(p)&j^3d(p)}\,\right|\,p\in U,\ 
h_1(p)=h_2(p)=h_3(p)=h_4(p)=0\Biggr\},\\
&W_2=\Biggl\{\left.\pmt{j^3a(p)&j^3b(p)\\
     j^3c(p)&j^3d(p)}\,\right|\,p\in U,\ 
h_1(p)=h_2(p)=h_3(p)=h_5(p)=0\Biggr\},
\end{align*}
where
$h_1=ad-bc$,
$h_2=(ad-bc)_w$,
$h_3=(ad-bc)_{ww}$,
$h_4=(ad-bc)_{www}$, and
$$
h_5=(ad-bc)_{u}(ad-bc)_{vw}-(ad-bc)_{v}(ad-bc)_{uw}.
$$
We define two functions
$H_i:J^3(U,E|_U)\setminus (Z\cup D)\to\R^4$
$(i=1,2)$ by
$H_1=(h_1,h_2,h_3,h_4)$, $H_2=(h_1,h_2,h_3,h_5)$.
Then it is sufficient to show that
$(0,0,0,0)$ is a regular value of each $H_1$ and $H_2$.
We calculate the derivative of $H_1$ 
with respect to the $16$ coordinates of $J^3(U,E|_U)$ 
corresponding to the zero, first, second and third derivatives 
by $\partial_w$ of $a,b,c,d$.
The matrix representation of them is
$$
\left(
\begin{array}{c|cccc}
{\cal M}&\\
\hline
*&d&-c&-b&a
\end{array}
\right),\quad
{\cal M}=
\left(
\begin{array}{cccccccccccccccc}
d&-c&-b&a& &  &  & & &  &  & \\
*& *& *&*&d&-c&-b&a& &  &  & \\
*& *& *&*&*& *& *&*&d&-c&-b&a\\
\end{array}
\right)
$$
where the blank entries are $0$.
Since $(a,b,c,d)\ne(0,0,0,0)$, we have the assertion for
$H_1$.
Next we calculate the derivative of $H_2$ 
with respect to the $20$ coordinates of $J^3(U,E|_U)$ 
corresponding to the zero, first, second derivatives 
by $\partial_w$ of $a,b,c,d$ and
corresponding to the derivatives by $\partial_u$, $\partial_w$,
and $\partial_v$, $\partial_w$ of $a,b,c,d$.
The matrix representation of them is
$$
\left(
\begin{array}{c|ccccccccc}
{\cal M} &  &  & & &  &  & \\
\hline
*&d(h_1)_v&-c(h_1)_v&-b(h_1)_v&a(h_1)_v
&
d(h_1)_u&-c(h_1)_u&-b(h_1)_u&a(h_1)_u\\
\end{array}
\right),
$$
Since $(a,b,c,d)\ne(0,0,0,0)$, 
and $(h_1)_u(p)=(h_1)_v(p)=h_2(p)=0$ means
$d \det\phi(p)=(0,0,0)$,
we have the assertion for
$H_2$.
\end{proof}

\section{Morin singularities from a manifold with tangent distribution}
Let $D_1$ be a rank $r$ tangent distribution on $M$,
and let $N$ be an $r$-dimensional manifold,
and $f:M\to N$ a map.
Setting $D_2=f^*TN$ and $\phi:D_1\to D_2$ 
by
$$\phi(v)=df(v),$$
we obtain a bundle homomorphism
between $D_1$ and $D_2$.
We call the above $\phi$ a {\it bundle homomorphism
induced by\/ $f$}.
In this section, assuming $f$ be a Morin singularity,
we consider relationships of
$\phi$, $D_1$ and $f$
in the case of $m=3$, $r=2$.
Moreover, we assume that
$M$ is an open neighborhood of $0$ in $\R^3$, 
$N$ is an open neighborhood of $0$ in $\R^2$,
and $f:(\R^3,0)\to(\R^2,0)$.
\subsection{Morin singularities}
We give a belief review on the Morin singularities
of
$(\R^3,0)\to(\R^2,0)$.
The map-germ $f:(\R^3,0)\to(\R^2,0)$ is called a
{\it definite fold\/} 
(respectively, a {\it indefinite fold\/}) if it is
${\cal A}$-equivalent to
the map-germ $(u,v,w)\mapsto(u,v^2+w^2)$ 
(respectively, $(u,v^2-w^2)$) at $0$.
Two map-germs $f,g:(\R^m,0)\to(\R^n,0)$ are
{\it ${\cal A}$-equivalent\/} if there exist
diffeomorphism-germs
$\sigma:(\R^m,0)\to(\R^m,0)$
and
$\tau:(\R^n,0)\to(\R^n,0)$
such that $\tau\circ f\circ\sigma^{-1}=g$.
The map-germ $f:(\R^3,0)\to(\R^2,0)$ is called a
{\it cusp\/}
if it is 
${\cal A}$-equivalent to
the map-germ $(u,v,w)\mapsto(u,v^2+w^3+uw)$.
Definite fold, indefinite fold and cusp are called
Morin singularities, and
it is known that generic 
singularities appearing on maps from a $3$-manifold
to a $2$-manifold are only
Morin singularities.
A characterization of Morin singularities is given
as follows:
Let $f:(\R^3,0)\to(\R^2,0)$ be a map-germ
and $\rank df_0=1$.
Then there exists a pair of vector fields 
$\{\xi, \eta_1,\eta_2\}$ such that
$$
\spann{\xi(0),\eta_1(0),\eta_2(0)}=T_0\R^3,\quad
\spann{\eta_1,\eta_2}=\ker df_p,\quad
p\in S(f),
$$
where $S(f)$ is the set of singular points of $f$.
We set 
$$
\lambda_1=\det(\xi f,\eta_1f),\quad
\lambda_2=\det(\xi f,\eta_2f),\quad
H=\pmt{
\eta_1\lambda_1&\eta_2\lambda_1\\
\eta_1\lambda_2&\eta_2\lambda_2}
$$
Then 
$f$ at $0$ is a definite fold (respectively, indefinite fold)
if and only if
$\det H(0)>0$ (respectively, $\det H(0)<0$).
We assume that $\rank H(0)=1$, then there exists
a vector field $\theta=a_1\eta_1+a_2\eta_2$ on $S(f)$
such that $\spann{\theta_0}=\ker H(0)$.
Then 
$f$ at $0$ is a cusp
if and only if
$\theta H(0)\ne0$.
See \cite{low} in detail.
\subsection{Conditions for singularities}
We take a frame $\{e_1,e_2\}$ of $D_1$.
We regard 
$e_1$, $e_2$ 
as vector fields.
We consider the conditions of singular points of
fold-like, cusp-like and swallowtail-like
singular points under the assumption
that $f$ is regular, fold and cusp
since these are generic singular points.

When $f$ is regular at $0$, and $D_1\not\subset \ker df_0$,
then $\phi$ is non-singular.
When $f$ is singular at $0$, and
$D_1\subset \ker df_0$,
then $\phi$ is of rank zero at $0$.
Since we are stick to rank one singular points of $\phi$,
we assume that $D_1\cap \ker df_0$ is one-dimensional.
By changing frame, we may assume that $e_1f(0)\ne0$.
The bundle homomorphism
$\phi$ can be represented by
the matrix
$$
(e_1f,e_2f)
$$
by $\{e_1,e_2\}$ and the trivial frame on $\R^2$.
Since $\rank \phi=1$ at $0$,
we take a null section $\eta_\phi$,
and set
$$
\lambda_\phi=\det(e_1f,e_2f)=\det(e_1f,\eta_\phi f).
$$
The following proposition holds.
\begin{proposition}\label{prop:hantei}
The singular point\/ $p$ of\/ $\phi$ is fold-like singular point
if and only if\/
$\det(e_1f,\eta_\phi^2f)\ne0$ at\/ $p$.
A non-degenerate singular point\/ $p$ is 
cusp-like singular point\/
$($respectively, swallowtail-like singular point\/$)$ 
if and only if\/
$\det(e_1f,\eta_\phi^2f)=0$, and\/
$\det(e_1f,\eta_\phi^3f)\ne0$ at\/ $p$
$($respectively, 
$\det(e_1f,\eta_\phi^2f)=\det(e_1f,\eta_\phi^3f)=0$,
$\det(e_1f,\eta_\phi^4f)\ne0$,
and\/
$d\det\big(\det(e_1f,\eta_\phi  f),
       \det(e_1f,\eta_\phi^2f),
       \det(e_1f,\eta_\phi^3f)\big)\ne0$ at\/ $p)$.
\end{proposition}
\begin{proof}
Since $\eta_\phi f(p)=0$,
it is obvious that the assertion for the fold-like singular point.
Let $p$ be a non-degenerate singular point,
and $\eta_\phi\lambda(p)=0$.
Since $\eta_\phi f=0$ on $S=\{\lambda_\phi=0\}$,
and $p$ is non-degenerate,
there exists a vector valued function $g$ such that
$\eta_\phi f=\lambda_\phi g$.
Then by the assumption $\eta_\phi\lambda(p)=0$,
$\eta_\phi^2 f(p)=0$.
Hence $\eta_\phi^2\lambda_\phi(p)\ne0$
is equivalent to $\det(e_1f,\eta_\phi^3f)(p)\ne0$.
This proves the assertion for the cusp-like singular point.
Next we assume that
$p$ be a non-degenerate singular point,
and $\eta_\phi\lambda(p)=\eta_\phi^2\lambda(p)=0$.
Then by the same reason as above, 
we have $\eta_\phi^2 f(p)=\eta_\phi^3 f(p)=0$.
Thus we see that
$\eta_\phi^3\lambda_\phi(p)\ne0$
is equivalent to $\det(e_1f,\eta_\phi^4f)(p)\ne0$,
and
$\det(\lambda_\phi,\eta_\phi\lambda_\phi,\eta^2_\phi
\lambda_\phi)(p)\ne0$
is equivalent to
$d\det\big(\det(e_1f,\eta_\phi  f),
       \det(e_1f,\eta_\phi^2f),
       \det(e_1f,\eta_\phi^3f)\big)(p)\ne0$.
This proves the assertion.
\end{proof}
\subsection{Restriction of singularities of $\phi$ by singular types of $f$}
We assume that $f$ at $0$ is a definite fold singular point.
Then
$\rank (e_1f,e_2f,e_3f)=1$ on $S(f)$,
where $\{e_1,e_2,e_3\}$ is a frame of $T\R^3$.
Thus
there exist functions $k_1,k_2$ such that
$e_2f=k_1e_1f,e_3f=k_2e_1f$ on $S(f)$.
Taking extensions
of
$k_1,k_2$ 
on $\R^3$, we set
$$
\eta_2=-k_1e_1+e_2,\quad
\eta_3=-k_2e_1+e_3,
$$
and also set
$$
\lambda_2=\det(e_1f,e_2f)=\det(e_1f,\eta_2f),
\quad\lambda_3=\det(e_1f,e_3f)=\det(e_1f,\eta_3f).
$$
Then we see that
$\eta_2$ is a null section of $\phi$,
and $\lambda_2$ is the same as
$\lambda_\phi$.
Since 
$f$ is definite fold,
$$
H=\det\pmt{
\eta_2\lambda_2&\eta_3\lambda_2\\
\eta_2\lambda_3&\eta_3\lambda_3}
>0.
$$
In particular, 
$\eta_2\lambda_2\ne0$.
Thus 
$\phi$ is fold-like at $0$ if $\rank \phi(0)=1$.

Next
we assume that $f$ at $0$ is a cusp singular point.
Then
we take $k_1,k_2,\eta_2,\eta_3$ and $\lambda_2,\lambda_3$
as above.
We assume that $\phi$ is not fold-like,
namely, $\eta_2\lambda_2(0)=0$.
Then since
$f$ is cusp,
$$
H(0)=\det\pmt{
\eta_2\lambda_2&\eta_3\lambda_2\\
\eta_2\lambda_3&\eta_3\lambda_3}(0)
=0.
$$
Since $\eta_3\lambda_2(0)=\eta_2\lambda_3(0)$,
it holds that $\eta_3\lambda_2(0)=0$.
Hence the kernel of $H$ is $\theta=\eta_1$ at $0$.
Then $f$ is cusp if and only if 
$$
\eta_1^2\lambda_1(0)\ \eta_2\lambda_2(0)\ne0.
$$
Thus 
$\phi$ is non-degenerate and not fold-like at $0$,
then $\phi$ is cusp-like at $0$.

\subsection{The case $D_1$ is a foliation}
In this section, we assume $D_1$ is a foliation.
By taking a coordinate system $(x,y,z)$ on $\R^3$,
we may assume $D_1=\spann{e_1,e_2}=\spann{\partial_x,\partial_y}$.
Let $L(x,y)$ be the leaf which contains the origin,
namely, $L(x,y)=f(x,y,0)$.
We have the following proposition.
\begin{proposition}\label{prop:folsing}
Under the above setting, the following holds\/{\rm :}
{\rm (1)} 
$\phi$ is fold-like if and only if\/ $L$ is fold.
{\rm (2)} if\/ $\phi$ is non-degenerate, then\/
$\phi$ is cusp-like if and only if\/ $L$ is cusp.
{\rm (3)} if\/ $\phi$ satisfies that\/
$\rank d(\lambda_\phi,\eta_\phi\lambda_\phi)(0)=2$,
then\/
$\phi$ is swallowtail-like if and only if\/ $L$ is swallowtail.
\end{proposition}
A map-germ $f:(\R^2,0)\to(\R^2,0)$ 
is called a {\it fold\/} if $f$ is ${\cal A}$-equivalent to 
the map-germ $(u,v)\mapsto(u,v^2)$ at $0$.
Furthermore, a map-germ $f:(\R^2,0)\to(\R^2,0)$ 
is called a {\it cusp\/} (respectively, {\it swallowtail\/})
if $f$ is ${\cal A}$-equivalent to 
$(u,v)\mapsto(u,v^3+uv)$ at $0$
(respectively, $(u,v)\mapsto(u,v^4+uv)$ at $0$).
Criteria for these singularities are obtained as follows:
Let $f:(\R^2,0)\to(\R^2,0)$ be a map-germ.
We set $\lambda=\det J$, where $J$ is the Jacobian matrix of $f$.
A singular point $p\in S(f)$ is {\it non-degenerate\/} if 
$d\lambda(p)\ne0$.
Then the following holds.
\begin{fact}\label{fact:plpl}{\rm \cite{w,suy3,cripl}}
A singular point\/ $p$ is fold
if\/ $\eta\lambda(p)\ne0$.
Moreover, a non-degenerate singular point\/ $p$ is cusp\/
$($respectively, swallowtail\/$)$
if\/ $\eta\lambda(p)=0$ and\/ $\eta^2\lambda(p)\ne0$
$($respectively, 
$\eta\lambda(p)=\eta^2\lambda(p)=0$ and\/ $\eta^3\lambda(p)\ne0)$.
\end{fact}
\begin{proof}[Proof of Proposition\/ {\rm \ref{prop:folsing}}]
We may assume that $f_x(0,0)\ne0$.
Then there exists a function $k_1(x,y,z)$ such that
if  $p\in S$, then
$$
f_y(p)=k_1(p)f_x(p).
$$
Take an extension of $k_1$ on $U$,
we take a null section 
$$
\eta_\phi=-k_1e_1+e_2.
$$
On the other hand,
there exists a function $l(x,y)$ such that
if  $q\in S(L)$, then
$$
f_y(q)=l(q)f_x(q).
$$
Take an extension of $l$ on $U\cap \{z=0\}$,
we take a null vector field of $L$
$$
\eta_L=-le_1+e_2.
$$
Set $\lambda_L(x,y)=\det(e_1f(x,y,0),\eta_Lf(x,y,0))$.
Then 
since $\lambda_\phi(x,y,0)=\lambda_L(x,y)$,
and 
$\eta_\phi(x,y,0)=\eta_L(x,y)$,
we see
\begin{equation}\label{eq:lambdafol}
\eta_\phi^i\lambda(0)=\eta_L^i\lambda_L(0)\quad
(i=1,2,3).
\end{equation}
The assertion is obvious by
Fact \ref{fact:plpl} and \eqref{eq:lambdafol}.
\end{proof}
\subsection{The case $D_1$ is a contact structure}
In this section, we assume $D_1$ is a contact structure.
Since the Hamilton vector field $X$ of $\lambda_\phi$
is contained in $D_1$
on $S$,
we consider the relationship with the behavior of $X$
and the singularities of $\phi$.
We may assume 
$D_1=\spann{e_1,e_2}=\spann{\partial_x,\partial_y-x\partial_z}$
without loss of generality.
Since $\phi$ can be expressed by
$(f_x,f_y-x f_z)$,
$$
\lambda_\phi
=
\det(f_x,f_y-x f_z).
$$
The Hamilton vector field
$X$
of $\lambda_\phi$ is
$$
X=
(\lambda_y-x\lambda_z)\partial_x
-\lambda_x\partial_y
-(\lambda-x\lambda_x)\partial_z
=(\lambda_y-x\lambda_z)e_1-\lambda_xe_2-\lambda\partial_z
.
$$
Since $S=\{\lambda_\phi=0\}$ holds,
$X_p\in D_1$ is equivalent to $p\in S$.
We have the following theorem.
\begin{theorem}\label{thm:conthamil}
If\/ $\phi$ has a corank one singular point at\/ $p$,
under the above setting,
$p\in S$ is fold-like if and only if
$$
X_p\quad\text{and}\quad(\eta_\phi)_p
$$
are linearly independent,
where\/ $\eta_\phi$ is a null section of\/ $\phi$.
\end{theorem}
\begin{proof}
Since $\phi$ is a corank one singular point at $p$,
there exist functions $k_1,k_2$ on $S$
such that
$(k_1,k_2)\ne(0,0)$ and
$k_1e_1f+k_2e_2f=0$.
Expanding $k_1,k_2$ to a neighborhood of $p$, 
we can take a null section 
$\eta_\phi=k_1e_1+k_2e_2$.
Then 
$$
\eta_\phi\lambda_\phi
=
k_1e_1\lambda_\phi+k_2e_2\lambda_\phi
=
k_1\lambda_x+k_2(\lambda_y-x\lambda_z)
=\det\pmt{
k_1&-(\lambda_y-x\lambda_z)\\
k_2&\lambda_x}
$$
shows the assertion.
\end{proof}
By Theorem \ref{thm:conthamil},
on the set of non-fold-like singular points,
$X$ is parallel to the null vector field,
by Propositions \ref{prop:foldcugeom} and \ref{prop:swgeom},
we have the following corollary.
\begin{corollary}
If $p\in S$ is a cusp-like singular point,
then $X_p\not\in T_pS_2$.
If $p\in S$ is a swallowtail-like singular point.
Then $$
\tilde\mu(0)=0,\quad\tilde\mu'(0)\ne0,
$$
where 
$\gamma(t)=(\gamma_1(t),\gamma_2(t))$ $(\gamma(0)=p)$ 
is a parameterization
of $S_2$,
and $\eta_{\gamma(t)}=a(t)\partial_u+b(t)\partial_v$,
and
$$
\tilde\mu(t)=\pmt{
\gamma_1(t)&a(t)\\
\gamma_2(t)&b(t)}.
$$
\end{corollary}
%\section{thebibliography}

\medskip

{\small
\begin{flushright}
\begin{tabular}{l}
\begin{tabular}{l}
(Saji)\\
Department of Mathematics,\\
Graduate School of Science, \\
Kobe University, \\
Rokkodai 1-1, Nada, Kobe 657-8501, Japan\\
  E-mail: {\tt sajiO\!\!\!amath.kobe-u.ac.jp}
\end{tabular}
\\
\\
\begin{tabular}{l}
(Tsuchida)\\
Department of Mathematics, \\
Hokkaido University,\\
Sapporo 060-0810, Japan \hspace*{26mm}\\
{\tt asahi-tO\!\!\!amath.sci.hokudai.ac.jp}
\end{tabular}
\end{tabular}
\end{flushright}}


\begin{thebibliography}{9}
\bibitem{a}
A. Honda,
{\it Isometric immersions with singularities between 
space forms of the same positive curvature},
arXiv:1508.07223.

%\bibitem{mather2}
%J. N. Mather,
%{\it Stability of\/ $C^{\infty}$ mappings. II. 
%Infinitesimal stability implies stability},
%Ann. of Math. (2) {\bf 89} (1969), 254--291.
%
%\bibitem{GG}
%M. Golubitsky and V. Guillemin,
%{\itshape Stable mappings and their singularities},
%Graduate Texts in Math. {\bf 14} Springer-Verlag, New York-Heidelberg, 1973.

\bibitem{krsuy}
M. Kokubu, W. Rossman, K. Saji, M. Umehara and
K. Yamada,
{\it Singularities of flat fronts in hyperbolic space},
Pacific J. Math. {\bf 221} (2005), no. 2, 303--351.

\bibitem{cripl}
K. Saji,
{\it Criteria for singularities of smooth maps from 
the plane into the plane and their applications},
Hiroshima Math. J. {\bf 40} (2010), no. 2, 229--239.

\bibitem{low}
K. Saji,
{\it Criteria for Morin singularities for maps into lower dimensions, 
and applications},
Real and complex singularities, 315--336, 
Contemp. Math., {\bf 675}, Amer. Math. Soc., Providence, RI, 2016.

\bibitem{suy1}
K. Saji, M. Umehara and K. Yamada,
{\it The geometry of fronts},
Ann. of Math. (2) {\bf 169} (2009), no. 2, 491--529.


\bibitem{suy3}
K. Saji, M. Umehara and K. Yamada,
{\it $A_k$ singularities of wave fronts},
Math. Proc. Cambridge Philos. Soc. {\bf 146} (2009), no. 3, 731--746. 


\bibitem{coh}
K. Saji, M. Umehara and K. Yamada,
{\it Coherent tangent bundles and Gauss-Bonnet formulas for wave fronts},
J. Geom. Anal. {\bf 22} (2012), no. 2, 383--409.

\bibitem{index}
K. Saji, M. Umehara and K. Yamada,
{\it An index formula for a bundle homomorphism of the 
tangent bundle into a vector bundle of the same rank, 
and its applications},
J. Math. Soc. Japan {\bf 69} (2017), no. 1, 417--457. 

\bibitem{trans}
T. Schmah and C. Stoica,
{\it Saari's conjecture is true for generic vector fields},
Trans. Amer. Math. Soc. {\bf 359} (2007), no. 9, 4429--4448. 

\bibitem{w}
H. Whitney,
{\it On singularities of mappings of euclidean spaces. I. 
Mappings of the plane into the plane},
Ann. of Math. (2) {\bf 62} (1955), 374--410.
\end{thebibliography}
\end{document}